\documentclass{amsart}

\usepackage{amssymb, amscd, graphicx, psfrag}

\newtheorem{thm}{Theorem}[section]
\newtheorem{cor}[thm]{Corollary}
\newtheorem{lem}[thm]{Lemma}
\newtheorem{problem}[thm]{Problem}

\theoremstyle{remark}

\theoremstyle{definition}

\newtheorem{alg}[thm]{Algorithm}

\numberwithin{equation}{section}
\numberwithin{thm}{section}

\newcommand{\e }{\varepsilon }
\newcommand{\Comp}{{\rm Comp\,}}
\newcommand{\abs}[1]{\vert #1 \vert}
\newcommand{\inv}{^{-1}}
\newcommand{\W}[1]{W^{( #1 )}}

\begin{document}






\title{exponentially generic subsets of groups}


%

\author{Robert Gilman}
\address{Department of Mathematical Sciences\\Stevens Institute of Technology\\Hoboken, NJ 07030}
\email{rgilman@stevens.edu}

\author{Alexei Miasnikov}
\address{Department of Mathematical Sciences\\Stevens Institute of Technology\\Hoboken, NJ 07030}
\email{amiasnikov@gmail.com}

\author{Denis Osin}
\address{Mathematics Department\\
Vanderbilt University\\
Nashville, TN 37240}
\email{denis.v.osin@vanderbilt.edu}


%


%

\begin{abstract}In this paper we study the generic, i.e., typical, behavior of finitely generated subgroups of hyperbolic groups and also the generic behavior of the word problem for amenable groups. We show that a random set of elements of a nonelementary word hyperbolic group is very likely to be a set of free generators for a nicely embedded free subgroup. We also exhibit some finitely presented amenable groups for which the restriction of the word problem is unsolvable on every sufficiently large subset of words. 
\end{abstract}

\maketitle



\maketitle

\section{Introduction}

Natural sets of algebraic objects are often unions of two unequal parts, the larger part consisting of generic objects whose structure is uniform and relatively simple, and the smaller including exceptional cases which have much higher complexity and provide most of resistance to classification. The essence of this idea first appeared in the form of zero-one laws in probability, number theory, and combinatorics.  In finite group theory the idea of genericity can be traced to a series of papers by Erd\H{o}s and Turan in 1960-70's (for recent results see~\cite{shalev}), while in combinatorial group theory the concept of generic behavior is due to Gromov. His inspirational  works \cite{Gro,Gro2,Gro3} turned the subject into an area of very active research, see for example~\cite{AO,A1,A2,BMR1,BMR2,BMR3,BogV,BV,BMS,CERT,BM,BV,Champetier1,Champetier2, Jit, KMSS1,KMSS2,KSS,KRSS,KR,Ollivier,Olsh,Rom,MTV,Zuk}.  

We mention in particular the remarkable results due to Kapovich and Schupp on generic properties of one-relator groups \cite{KS1,KS2} and by Maher \cite{maher} and Rivin \cite{Rivin} on generic properties of random elements of mapping class groups and automorphisms of free groups, as well as the theorem by Kapovich, Rivin, Schupp and Shpilrain that generic cyclically reduced elements in free groups are of minimal length in their automorphic orbits \cite{KRSS}. An earlier  series of papers \cite{Olsh,AO,A1,A2}  by Arjantseva and Olshanskii established the theory of subgroups of random groups and related questions. 

Knowledge of  generic properties of objects can be used in design of  simple practical algorithms that work very fast on most inputs. In cryptography, several successful attacks have exploited  generic properties of randomly chosen  objects to break cryptosystems \cite{MU,MSUbook,MSU,RST}. Explicit generic case analysis of algorithmic problems first appeared in the papers~\cite{KMSS1,KMSS2,BMR1}. 

In the first part of this paper we show that with high probability a random subgroup of a  nonelementary hyperbolic group has a simple structure and is embedded without much distortion of its intrinsic metric. Arbitrary subgroups on the other hand, can be very complicated. A remarkable  construction introduced by Rips~\cite{Rips} shows that every finitely presented group $G$ is a quotient of a hyperbolic (in fact small cancellation) group $H$ by a finitely generated normal subgroup $N$. The Dehn function of $G$ is intimately related to the metric distorsion of the subgroup $N$ in $H$. In particular, as Rips noticed, the membership problem for $N$ in $H$ is undecidable provided the word problem in $G$ is undecidable. A host of undecidability results for subgroups of hyperbolic groups has been proven by combining the Rips technique with known unsolvability results for finitely presented groups  (\cite{BMS}, \cite{BW}). These results show that hyperbolic groups contain finitely generated subgroups with as much distortion as one pleases. However, it is widely believed that such subgroups are rare, and that most finitely generated subgroups of hyperbolic groups have an uncomplicated structure and not much distortion.

We prove here that for each $k \geq 1$, with overwhelming probability (relative to a natural distribution)  $k$-tuples of words in a given finite  set of generators of a non-elementary hyperbolic group freely generate a free subgroup which is quasi-isometrically embedded into the ambient group. The property that a random $k$-tuple of words is, with overwhelming probability, a set of free generators is sometimes referred to as the  generic Nielsen property.  In \cite{MU} Myasnikov and Ushakov proved that similar results hold in pure braid groups, as well as right angled Artin groups. This result has been applied to a rigorous mathematical cryptanalysis of the Anshel-Anshel-Goldfeld public key exchange scheme \cite{AAG}, including an analysis of various length-based attacks (\cite{HT}, \cite{GK}, \cite{RST}).  For free non-abelian groups the generic Nielsen property was shown earlier in \cite{Jit} and \cite{MTV}. Notice, that in the case of free groups, all finitely generated subgroups are free and embedded quasi-isometrically.

For related results on free products with amalgamation and HNN extensions we refer to \cite{FMR}. Beyond cryptographic applications our results on generic subgroups in hyperbolic groups provide a cubic time deterministic partial algorithm $\mathcal{A}$, which never lies and solves the membership problem for almost all (more precisely for a certain exponentially generic subset $\mathcal{D}$) of finitely generated subgroups in a given non-elementary hyperbolic group. Furthermore, if a given  subgroup is not in the set $\mathcal{D}$ the algorithm quickly recognizes this (in quadratic time) and halts with a failure message.

Another result we would like to mention here concerns with the complexity of the word problem in finitely presented groups. It turns out that many famous undecidable problems are, in fact, very easy on generic set of inputs. This is precisely the case for the halting problem of Turing machines with one-ended infinite tape~\cite{HM}, and for the classical examples of finitely presented groups or semigroups with undecidable word problem \cite{MUW}. The first examples of finitely presented semigroups where the word problem is undecidable on any generic set of inputs (words in the given set of generators) are constructed in \cite{MR}. Whether there exist such examples in finitely presented groups is still an open problem. In this paper we describe some finitely presented groups for which the word problem is undecidable on any exponentially generic set of words in given generators. The famous construction \cite{Kh}, due to Kharlampovich, of finitely presented solvable groups with undecidable word problem provide a host of examples of such groups.

In the next section we describe our main results in detail, and prove them in the following sections. The last section contains several open problems in this area.

\section{Statement of results}

Fix a finite alphabet with formal inverses, $A=\{ a_1, \ldots , a_m, a_1\inv, \ldots, a_m\inv \}$ for some $m\ge 2$. Use $\abs w$ to denote the length of a word $w$ over $A$ and $\abs S$ for the cardinality of a set $S$. Formal inverses, $w\inv$, are defined in the obvious way.

By $W$ we denote the free monoid with basis $A$, that is, the set of all words over the alphabet $W$ with the binary operation of concatenation. The subset $W_n = \{w \in W \mid \abs{w} \le n\}$ is the disk of radius $n$ in $W$, and $W = \cup_{n = 1}^\infty W_n$ is the stratification of $W$ by disks. Since every disk is finite, one may define the standard uniform distribution $\mu_n$ on $W_n$. The ensemble of distributions $\{\mu_n\}$, after a proper normalization, induces the standard ''uniform distribution`` $\mu$ on $W$ relative to the stratification by disks.

The exponential asymptotic density of $X\subset W$ is defined as
\[\rho_e(X)=\lim_{n\to\infty}\frac{|X\cap W_n|}{|W_n|}\]
if the limit converges exponentially fast.  In other words
$$\rho_e(X)= \lambda  \Longleftrightarrow |\lambda - \frac{|X\cap W_n|}{|W_n|}| \le \alpha^n$$
 for some constant $\alpha \in (0,1)$ and all sufficiently large $n$,
or equivalently if
$$|\lambda - \frac{|X\cap W_n|}{|W_n|}|  \le M\beta^n$$
 for some $\beta \in (0,1)$, positive constant $M$ and all $n$.

$X\subset W$ is {\em exponentially generic} if $\rho_e(X)=1$ and {\em exponentially negligible} if its complement is exponentially generic, i.e., if $\rho_e(X)=0$. It is clear that finite intersections of exponentially generic sets are exponentially generic and finite unions of exponentially negligible sets are exponentially negligible. See \cite{BMS,BMR1} for more information on asymptotic density.

To study asymptotic properties of $k$-generated subgroups of groups generated by $A$ we need to extend the notions introduced above to subsets of  $k$-tuples of words from $W$. For $k\ge 1$ put
\begin{equation}\label{eq:ktuples}
\W k = \{(w_1,\ldots, w_k) \mid w_i \in W\}.
\end{equation}
 The disk of radius $n$ in $\W k$ is defined to be
\begin{equation}\label{eq:ktuple-disks}
\W k_n  = \overbrace{W_n\times \cdots \times W_n}^k = \{(w_1, \ldots,w_k)  \in\W k\mid \abs{w_i} \le n\}.
\end{equation}
Exponential asypmtotic density of subsets of $\W k$ is defined as above but with $\W k_n$ in place of $W_n$. When $k$ is fixed or irrelevant, we write
\[ \vec w  \mbox{ for }(w_1, \ldots,w_k)\mbox{,  and }\abs{\vec w}\mbox{ for }  \max\{\abs{w_i} \mid i = 1, \ldots,k \}.\]

For any group $G$  a monoid epimorphism $W \to G$ which respects inverses is called a choice of generators for $G$, and the image in $G$ of $w\in W$ is denoted $\overline w$. Each choice of generators determines a word metric with distance $\abs{g-h}$ equal to the length of the shortest word in $W$ representing $g\inv h$. We abbreviate $\abs{g-1}$ as $\abs g$ or $\abs{g}_G$ if the ambient group is not clear. Note that for $w\in W$, $\abs w$ is the length of $w$ while $\abs{\overline w}$ is the length of the shortest word in $W$ mapping to $\overline w$.

Let $H$ be a finitely generated subgroup of $G$ with choice of generators $B^\ast \to H$ for some finite alphabet $B$ with formal inverses. The subgroup $H$ is {\em undistorted} in $G$ (with respect to the choices of  generators for $G$ and $H$) if it is {\em quasi-isometrically} embedded in $G$, i.e., there is a constant $\lambda > 1$ such that for every elements $f,h \in H$ the following inequality holds
$$ \frac{1}{\lambda} |f-h|_H \leq  |f-h|_G.$$

A nontrivial subgroup $H$ is undistorted if and only if the compression factor of $H$ in $G$ is positive. The compression factor of $H$ in $G$ (with respect to choices of generators $A\to G$ and $B\to H$) is defined as
\begin{equation}\label{eq:compression}
\Comp (G, A; H, B) = \inf\limits_{h\in H\setminus \{ 1\} }
\frac{|h|_G}{|h|_{H}}
\end{equation}
where
\[
|h|_{H}=\min\limits_{h=b_{i_1}\cdots b_{i_s}} (|b_{i_1}|_G +\cdots
+|b_{i_s}|_G),\]
and the minimum is taken over all representations
of $h$ in the form $b_{i_1}\cdots b_{i_s}$ with $b_{i_j}\in B$, $1\le j\le s$.

Recall that the {\em gross cogrowth} $\theta$ of $G$ with respect to a choice of generators $W \to G$ is
defined by
\begin{equation}\label{eq:cogrowth}
\theta = \lim_{n\to \infty}\frac{1}{2n} \log_{2m} |V_{2n}|
\end{equation}
where for any $r$, $V_r$ is the subset of all words of length $r$ in $W$ which represent the identity in $G$.
It is known (see Section \ref{subsec:amenable} for details and references) that $G$ is amenable if and only if $\theta = 1$.

The main technical result of the paper is Lemma~\ref{C}, which says that a certain set $\mathcal C\subset \W k$ which is defined in terms of a parameter $\e > 0$ is exponentially generic. The exponentially generic sets mentioned in the next two theorems all contain $\mathcal C$. Recall that a group is called {\em elementary} if it contains a cyclic subgroup of finite index.

\begin{thm}\label{hyp}
Let $G$ be a non-elementary hyperbolic group. Then  for any  choice of generators $W\to G$ the following sets are exponentially generic.
\begin{enumerate}
\item\label{hyp1} The set of all $(w_1, \ldots, w_k)\in W^{(k)}$ for which $\overline w_1, \ldots, \overline w_k$ generate a free subgroup of rank $k$ in $G$.
\item\label{hyp2} The set of all $(w_1, \ldots, w_k)\in W^{(k)}$ for which $\overline w_1, \ldots, \overline w_k$ generate a subgroup with compression factor at least $\frac{1-\theta}{\theta} -\e$, where $\theta $ is the gross cogrowth of $G$ with respect to the given choice of generators and $\e$ is any positive constant.
\end{enumerate}
\end{thm}

It is easy to see that the first statement of Theorem \ref{hyp} also applies to any group with choice of generators $W\to G$ which surjects onto a non-elementary hyperbolic group. Examples of such groups include many relatively hyperbolic groups, e.g., non-elementary groups hyperbolic relative to proper residually finite subgroups \cite{Osi07}. The later class includes fundamental groups of complete finite volume manifolds of pinched negative curvature, $CAT(0)$ groups with isolated flats, groups acting freely on $\mathbb R^n$--trees, and many other examples.

\begin{thm}\label{MP} Let $G$ be a non-elementary hyperbolic group. Then for any  choice of generators $W\to G$ and $k\ge 1$
there exists a partial algorithm $\mathcal A$ which for each  $\vec w = (w_1,\ldots,w_k)$ in an exponentially generic subset $\mathcal D\subset \W k$ and an arbitrary $z\in W$ decides if $\overline z$ is in the subgroup $H=\langle \overline w_1,\ldots, \overline w_k \rangle \subset G$. When the answer is yes, $\mathcal A$ decomposes $\overline z$ as a word in the generators $\overline w_1,\ldots, \overline w_k $ and their inverses. On all inputs $\mathcal A$ runs in time
$O((k|\vec w|+|z|)^3)$.
\end{thm}
By partial algorithm we mean one which never gives a wrong answer but may say ''Don't know`` or ''Fail``.

\begin{thm}\label{generic}
Let $G$ be a finitely presented amenable group with unsolvable word problem. Then for any choice of generators $W\to G$  the word problem in  $G$  is not solvable on any exponentially generic subset of $W$.
\end{thm}

As we noted above,~\cite{Kh} provides many groups to which Theorem~\ref{generic} applies.

\section{Preliminaries}

In this section we recall for convenience various known results and draw some elementary consequences. Recall the definitions of $W, W_n, \W k$ and $\W k_n$ from the preceding section.

\subsection{Asymptotic density}

\begin{lem}\label{spherical}
Define $I_n = \{w\in W\mid |w|=n\}$ (the sphere of radius $n$). If $\lim_{n\to\infty}\frac{|X\cap I_n|}{|I_n|} < \alpha^n$ for some $\alpha\in (0,1)$, and all sufficiently large $n$, then $X$ is exponentially negligible.
\end{lem}

\begin{proof}
Let $r$ be the greatest integer less than $n/2$.
\begin{align*}
\frac{|X\cap W_n|}{|W_n|} &\le \frac{|W_r|}{|W_n|} + \frac{|X\cap I_{r+1}| + \cdots + |X\cap I_n|}{|W_n|}\\
&\le (2m)^{-n/2} + \frac{|X\cap I_{r+1|}}{|I_{r+1}|} + \cdots + \frac{|X\cap I_n|}{|I_n|}\\
&\le (2m)^{-n/2} + \alpha^{r+1} + \cdots + \alpha^n\mbox{ for $n$ sufficiently large}\\
&\le (2m)^{-n/2} + \frac{\alpha^{n/2}}{1-\alpha}
\end{align*}
\end{proof}
Concatenation of all entries of $\vec w = (w_1,\ldots, w_k)\in \W k$ defines a map $\pi:\W k\to W$. It is easy to see that that  $\pi(\W k_n)= W_{nk}$ whence $|\W k_n|\ge |W_{nk}|$. The $k$-tuples in $\pi\inv(w)$ correspond to ordered partitions $\ell_1+\cdots+\ell_k=|w|$ with $0\le \ell_i\le |w|$. There are at most $(|w|+1)^k$ such partitions, and it follows that the restriction of $\pi$ to $\W k_n$ is at most $(nk+1)^k$ to~$1$. These conclusions still apply if we pick a fixed sequence of exponents $e_1,\ldots, e_k$ with $e_i=\pm 1$ and define $\pi(\vec w)= w_1^{e_1}\cdots w_k^{e_k}$.

\begin{lem}\label{projection} Define $\pi:\W k\to W$ by $\pi(\vec w) =  w_1^{e_1}\cdots w_k^{e_k}
$ as above. If $\pi(X)$ is exponentially negligible, then so is $X$.
\end{lem}

\begin{proof}
If $\pi(X)$ is exponentially negligible, then $\frac{|\pi(X)\cap W_{kn}|}{|W_{kn}|}\le \alpha^n$ for some $\alpha\in (0,1)$ and all sufficiently large $n$. Thus
\[\frac{|X\cap \W k_n|}{|\W k_n|} \le \frac{(nk+1)^k|\pi(X)\cap W_{kn}|}{|W_{kn}|} \le (nk+1)^k\alpha^n \]
and a straightforward argument shows that $X$ is exponentially negligible.
\end{proof}

\subsection{Amenable groups}
\label{subsec:amenable}

Let $W \to G$ be a choice of generators for a group $G$.  Define $V$ to be the subset of all words in $W$ which map to $1$ in $G$, and $V_n=V\cap I_n$ is the set of all words of length $n$ in $V$.

By~\cite{Grigorchuk1,Grigorchuk2} (see also \cite{Cohen} and \cite{Kes2}) $G$ is {\em amenable} if and only if
\[\limsup_{n\to\infty} (\abs{V_n}/\abs{I_n})^{1/n} = 1.
\]
Clearly  $|V_{n+p}|\ge |V_n||V_p|$, and $V_{2n}$ includes all concatentations of $n$ terms of the form $a_ia_i^{-1}$ or $a_i^{-1}a_i$. It follows that $|V_{2n}|\ge (2m)^n$; and if $|V_n|=0$, then $|V_{n-2}|=0$. Thus $|V_n|$ is positive for all even $n$ and either positive for all odd $n$ greater than than some bound $M$ or $0$ for all odd $n$.

In first case let $t=ks+r$ with $s > M$ and $M < r \le M+s$. Then $|V_t|\ge |V_s|^k|V_r|$ implies $|V_t|^{1/t}\ge |V_s|^{1/s}(|V_s|^{-r}|V_r|)^{1/t}$ whence $\liminf_{t\to\infty}|V_t|^{1/t}\ge |V_s|^{1/s}$. It follows that $\liminf_{t\to\infty}|V_t|^{1/t}\ge \limsup_{s\to\infty}|V_s|^{1/s}$, which in turn implies that $\lim_{n\to\infty}|V_n|^{1/n}$ exists. Likewise in the second case $\lim_{n\to\infty}|V_{2n}|^{1/(2n)}$ exists. Thus we may define
\begin{equation}\label{eq:spectralradius}
\lambda=\lim_{n\to\infty}(|V_{2n}|/|I_{2n}|)^{\frac{1}{2n}} = \frac{1}{2m}\lim_{n\to\infty}|V_{2n}|^\frac{1}{2n} = \frac{1}{2m}\limsup_{n\to\infty}|V_n|^{1/n}.
\end{equation}
$|V_{2n}|\ge (2m)^n$ implies $1 \ge \lambda \ge 1/\sqrt{2m}$. Comparison of~(\ref{eq:spectralradius}) with~(\ref{eq:cogrowth}) yields
\begin{equation}\label{eq:cogrowth2}
\theta=1+\log_{2m}\lambda = \limsup_{n\to \infty}\frac{1}{n} \log_{2m} |V_n|
\end{equation}
whence
\begin{equation}\label{eq:half}
1\ge \theta \ge 1/2.
\end{equation}
Thus amenability is equivalent to both $\lambda=1$ and $\theta=1$.

Also it follows from~\cite[Corollary 1, page 343]{Kes1} that every subgroup of an amenable group is amenable. Conversely a group which contains a non-amenable subgroup is itself nonamenable.

\begin{lem}\label{amenable} If $G$ is non-amenable, then for any $\epsilon > 0$ and  constant $K$, $U= \{w\in W \mid \abs{\overline w} > (\frac{1-\theta}{\theta}-\epsilon)\abs w  + K\}$ is exponentially generic.
\end{lem}

\begin{proof} First suppose $K=0$. Choose $\e>0$ and let $\rho = (2m)^{\theta + \e}$. As $m\ge 2$, (\ref{eq:half}) implies $\rho\ge 2$. If $w\in I_{n,r} =\{w\in I_n \mid \abs{\overline w} \le r\}$, then $ww'\in V_{n+s}$ for some $w'$ of length $s\le r$. Thus $I_{n,r} \subset V_n\cup\cdots\cup V_{n+r}$. For $n$ sufficiently large, (\ref{eq:cogrowth2}) yields
\begin{align*}
\abs{I_{n,r}} &\le \rho^n+\cdots +\rho^{n+r}\\
&=\rho^n\,\frac{\rho^{r+1}-1}{\rho - 1}\\
&\le \rho^{n+r}\frac{\rho}{\rho -1}\\
&\le 2\rho^{n+r}.
\end{align*}
Consequently $\frac{\abs{I_{n,r}}}{\abs{I_n}} \le 2(2m)^{(\theta+\e)(r+n) - n}$. If $r\le (\frac{1-\theta}{\theta}-\e)n$ then $(\theta+\e)((\frac{1-\theta}{\theta}-\e)+1) - 1 = -\e^2$ implies $\frac{\abs{I_{n,r}}}{\abs{I_n}} \le 2(2m)^{-\e^2}$ whence the complement of $U$ is exponentially negligible by  Lemma~\ref{spherical}.

Now suppose $K>0$. For any $\e>0$, $U'= \{w\in W \mid \abs{\overline w} \ge (\frac{1-\theta}{\theta}-\epsilon/2)\abs w\}$ is exponentially generic. But $w\in U'$ implies
\begin{align*}
\abs{\overline w} &\ge (\frac{1-\theta}{\theta}-\epsilon/2)\abs w\\
&\ge  (\frac{1-\theta}{\theta}-\epsilon)\abs w + \e/2|w|\\
&\ge  (\frac{1-\theta}{\theta}-\epsilon)\abs w + K
\end{align*}
for $|w|$ sufficiently large. Thus $U$ contains a co-finite subset of $U'$.
\end{proof}

\begin{lem}\label{quasi}
If $G$ is non-amenable, then for any $\e > 0$
\begin{enumerate}
\item\label{part1} The set of words $w$ with $\abs {\overline v} \ge (\frac{1-\theta}{\theta} -\e)\abs v$ for all subwords $v$ of $w$ with $\abs v \ge \e \abs w$ is exponentially generic;
\item\label{part2} The set of words $w$ with $\abs {\overline v} \ge (\frac{1-\theta}{\theta} -\e)\abs v$ for all subwords $v$ of $ww$ with $\abs{w} \ge \abs v \ge \e \abs w$ is exponentially generic.
\end{enumerate}
\end{lem}
\begin{proof}
Let $\rho=\frac{1-\theta}{\theta} -\e$.
The words in $I_n$ are obtained by filling a sequence of $n$ locations $\ell_1 \ldots, \ell_n$ with letters from $A$ in all possible ways. Fix $i$ and $j$ with $j-i+1\ge \e n$. It follows from the proof of Lemma~\ref{amenable} that for some $\alpha\in (0,1)$ and $n$ sufficiently large, the fraction of ways of filling the subsequence $\ell_i,\ldots, \ell_j$ with a word $v$ such that $\abs {\overline v} < \rho \abs v = \rho(j-i+1)$ is less than $\alpha^{|v|}$. Since each $v$ extends to $w\in I_n$ in $(2m)^{n-|v|}$ ways, $\alpha^{\abs v}$ also bounds the fraction of extensions which fail the condition at the subword $v$. There are $n^2$ choices of $i,j$, so we conclude that the fraction of words $w\in I_n$ which fail is at most $n^2\alpha^{\e n}$. Thus the first assertion holds by Lemma~\ref{spherical}. The second is proved similarly by counting the number of extensions of $v$ to $ww$. The condition $|w| \ge \abs v$ insures that $v$ extends to a word of the form $ww$ in $(2m)^{(n-\abs v)}$ ways.
\end{proof}

\subsection{Hyperbolic metric spaces}

Recall that a metric space $M$ is {\em geodesic} if distances between points are realized by geodesics, and a geodesic metric space is $\delta$-{\em hyperbolic} for some $\delta \ge 0$ (or simply {\em hyperbolic}) if any geodesic triangle $T$ in $M$ is {\em $\delta $-thin}. That is, each side of $T$ belongs to the union of the closed $\delta $--neighborhoods of the other two sides \cite{Gro}.

We denote  a geodesic path in $M$ from $p$ to $q$ by $[p,q]$ and  its length by  $\abs{p-q}$. The next lemma is well-known (see, e.g.,~\cite{GdH}).

\begin{lem}\label{metricspace}
\begin{enumerate}
\item For any geodesic quadrilateral with vertices $p,q,r,s$,
\[\abs{p-s} + \abs{q-r}\le 2\delta + \max\{\abs{p-q}+\abs{r-s}, \abs{p-r}+\abs{q-s}\}.
\]
\item For any geodesic triangle $T$ with vertices $p,q,r$, there are points $t_p, t_q, t_r$ on the sides opposite $p$, $q$, $r$ respectively such that
\begin{enumerate}
\item $t_p, t_q, t_r$ are a distance at most $2\delta$ from each other;
\item $\abs{p-t_q} = \abs{p-t_r}$, and likewise for the other vertices;
\item Points lying an equal distance from $p$ along the segments of the sides of $T$ from $p$ to $t_q$ and $p$ to $t_r$ are a distance at most $2\delta$ from each other. Similar statements hold for the other vertices.
\end{enumerate}
\end{enumerate}
\end{lem}
\noindent
The quantity $\abs{p-t_q} = \abs{p-t_r}$ is the {\em Gromov product} of $q$ and $r$ with respect to $p$, usually written  $(q|r)_p$. It is not hard to show that
\begin{equation}\label{eq:Gromov}
(q|r)_p=\frac12(\abs{p-q}+\abs{p-r}-\abs{q-r}).
\end{equation}
 Thus $(q|r)_p$ is independent of the choice of geodesics forming the sides of  a triangle with vertices $p,q,r$.

The following lemma improves~\cite[Theorem 16 in Chapter 5]{GdH}.
\begin{lem}\label{length}
If for some $\kappa > 0$ and $n\ge 2$, the points $p_0,\ldots, p_n$ satisfy
\begin{equation}
\abs{p_i - p_{i+2}}\ge \kappa+2\delta+\max\{\abs{p_i -p_{i+1}}, \abs{p_{i+i} -p_{i+2}}\}\label{e1}
\end{equation}
then
\[\abs{p_0-p_n}\ge \abs{p_0-p_{n-1}} + \kappa \ge \abs{p_0-p_{1}} + (n-1)\kappa \ge \kappa n. \]
\end{lem}
\begin{proof} The first inequality implies the second by induction, and the second implies the third as $|p_0-p_1| \ge \kappa$ lest the hypothesis fail for $i=0$. Thus it suffices to prove
\begin{equation}
\abs{p_0-p_n}\ge \abs{p_0-p_{n-1}} + \kappa. \label{f}
\end{equation}
Clearly (\ref{f}) holds when $n=2$; assume $n\ge 3$. By the first part of Lemma~\ref{metricspace},
\begin{eqnarray}
\abs{p_0-p_{n-1}} + \abs{ p_{n-2}-p_n} &\le& 2\delta + \abs{p_0-p_{n-2}}+\abs{p_{n-1}-p_n} \label{e2}\\
\abs{p_0-p_{n-1}} + \abs{ p_{n-2}-p_n} &\le& 2\delta + \abs{p_{n-2}-p_{n-1}}+\abs{p_0 - p_n}\label{e3}
\end{eqnarray}
By induction and (\ref{e1}), the lefthand side of (\ref{e2}) is greater than or equal to $\abs{p_0-p_{n-2}} + \kappa + 2\delta + \abs{p_{n-1}-p_n}$, which contradicts (\ref{e2}), as $\kappa > 0$. Consequently (\ref{e3}) holds. Applying (\ref{e1}) to the lefthand side of (\ref{e3}) yields $\abs{p_0-p_{n-1}} + \kappa + 2\delta \le 2\delta + \abs{p_0 - p_n}$ as desired.
\end{proof}

The following lemma is \cite[Proposition~1.6, Chapter III.H]{BH} and \cite[Lemma~1.5, Chapter 3]{CDP}.

\begin{lem}\label{log} Let $\gamma$ be a path of length $\ell$ from $p$ to $q$ in a $\delta$-hyperbolic space, and $[p,q]$ a geodesic from $p$ to $q$. Any point on $[p,q]$ is a distance at most $1+\log_2\ell$ from some point on $\gamma$.
\end{lem}

\begin{cor}\label{log2}
Let $\gamma$ be a path of length $\ell$ from $p$ to $q$ in a $\delta$-hyperbolic space, and $[p,q]$ a geodesic from $p$ to $q$. Any point on the first half of $[p,q]$ is a distance at most $1+2\delta+\log_2\ell$ from some point on the first half of $\gamma$.
\end{cor}

\begin{proof}
Choose a point $r$ on $\gamma$ so that a geodesic triangle $T$ with vertices $p,q,r$ is isoceles with base $[p,q]$. Because $T$ is isoceles, the vertex $t_r$ defined in Lemma~\ref{metricspace} is at the midpoint of $[p,q]$. It follows that any point on the first half of $[p,q]$ is a distance at most $2\delta$ from $[p,r]$. Apply Lemma~\ref{log} to the subpath of $\gamma$ from $p$ to $r$ and the geodesic $[p,r]$.
\end{proof}

\section{Subgroups of hyperbolic groups}

A group $G$ with choice of generators $W\to G$ is {\em $\delta$-hyperbolic} for some $\delta >0$ (or simply {\em hyperbolic}) if its Cayley graph $\Gamma$ (with edges isometric to the unit interval) is a $\delta $-hyperbolic metric space. The word metric on $G$ extends to a metric on $\Gamma$.

A hyperbolic group is called {\em elementary}  if it contains a cyclic subgroup of finite index.  As non-elementary hyperbolic groups contain nonabelian free subgroups~\cite{Del}, they are non-amenable. Throughout this section, $G$ denotes a non-elementary $\delta$-hyperbolic group.

For each word $w\in W$ and a  vertex $x$ in $\Gamma$, there is a unique path in $\Gamma$ with initial point $x$ and label $w$. Thus we will speak of the path $w$ starting at $x$; $w\inv$ is the same path traversed in the opposite direction starting at the endpoint of $w$.

\begin{figure}
\begin{center}
\psfrag{1}{$1$}
\psfrag{u}{$z$}
\psfrag{v}{$v$}
\psfrag{w}{$u$}
\psfrag{vv}{$\overline v$}
\psfrag{ww}{$\overline u\inv $}
\psfrag{t1}{$t_1$}
\psfrag{tv}{$t_{\overline v}$}
\psfrag{tw}{$t_{\overline u\inv}$}
\psfrag{r}{$r$}
\psfrag{s}{$s$}
\psfrag{p}{$p$}
\psfrag{q}{$q$}
\includegraphics[width=1.75in]{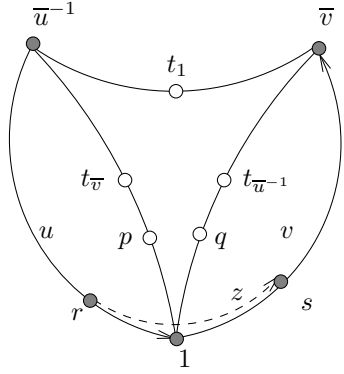}\\
\end{center}
\caption{The triangle $T$ from the proof of Lemma~\ref{prod2}. Shaded dots are vertices of the Cayley diagram of $G$\label{triangle}}
\end{figure}

\begin{lem}\label{prod2}
For any $\e > 0$, the set of $\vec w = (u, v) \in \W 2$ with $(\overline u^{\pm 1}| \overline v^{\pm 1} )_1 < \e \min \{ |u|, |v| \}$ is exponentially generic.
\end{lem}

\begin{proof}
Without loss of generality assume $\e< 1/2$. A straightforward counting argument shows that
the fraction of $(u,v)\in \W 2_n$ with $|u| < n/2$ or $|v| < n/2$ is less than $2(2m)^{-n/2}$. It follows that $\{\vec w\in \W 2 \mid \min\{|u|, |v|\} < |\vec w|/2\}$ is exponentially negligible. On the complementary set $\min\{|u|, |v|\}\ge |\vec w|/2$. Thus to complete the proof it suffices to show that
\[X=\{ \vec w\in \W 2_n \mid (\overline u^e | \overline v^f )_1 \ge \e |\vec w|/2 \}
\]
is exponentially negligible for each $e=\pm 1$ and $f=\pm 1$. Consider $e=f=1$; the other cases are similar.

For any $\vec w \in X$ let $T$ be a geodesic triangle in the Cayley diagram $\Gamma$ with vertices $1$, $\overline u\inv$ and $\overline v$ as in Figure~\ref{triangle}.
Pick points $p$ and $q$ a distance $\e |\vec w|/2$  from $1$ along the geodesics $[1, \overline u\inv  ]$ and $[1, \overline v]$ respectively. By Lemma~\ref{metricspace}  $\abs { p - q} \le 2\delta$.  As every point of $\Gamma$ is a distance at most $1/2$ from a vertex, Lemma~\ref{log} yields $\abs {p-r} \le 3/2 + \delta \log_2 |\vec w| $ for some vertex $r$ on $u$. Likewise $\abs {q-s} \le 3/2 + \delta \log_2 |\vec w| $ for some vertex $s$ on $v$.

Let $z$ be the subword of $w$ which labels the subpath from $r$ to $s$. By construction
\begin{align*}
\abs{\overline z} &= \abs {r-s} \le  3 + 2\delta + 2\delta\log_2 |\vec w| \\
\abs z &\ge (|p| - |p - r|) + (|q| - |q - s|) \ge \e |\vec w| - (3 + 2\delta\log_2 |\vec w|).
\end{align*}
As $|\vec w|\le |uv|\le 2|\vec w|$, we have
\begin{align*}
\abs{\overline z} &\le (\frac{1-\theta}{\theta}-\frac{\e}{4})|\vec w|\le  (\frac{1-\theta}{\theta}-\frac{\e}{4})|uv|\\
\abs z &\ge \frac{\e|\vec w|}{2} \ge \frac{\e}{4}|uv|
\end{align*}
for $|\vec w|$ large enough; that is, for all $\vec w$ in some co-finite subset $X'$ of $X$. By Lemma~\ref{quasi}(\ref{part1}) the image of $X'$ under the map $\pi$ of Lemma~\ref{projection} is exponentially negligible. By Lemma~\ref{projection} $X'$ and hence $X$ are exponentially negligible.
\end{proof}

\begin{lem}\label{prod1}
For any $\e > 0$, the set of $ w \in W$ such that $(\overline w| \overline w\inv )_1 < \e \abs w$ is exponentially generic.
\end{lem}
\begin{proof}
Form a geodesic triangle $T$ with vertices $1$, $\overline w\inv$, $\overline w$, and argue as in the proof of the preceding lemma but with Lemma~\ref{prod1} in place of Lemma~\ref{prod2} and
Lemma~\ref{quasi}(\ref{part2}) in place of  Lemma~\ref{quasi}(\ref{part1}).
\end{proof}

\begin{lem}~\label{C} Fix $\e\in (0,1)$. The set $\mathcal C$ of all $\vec w=(w_1,\ldots, w_k)\in \W k$ satisfying the following conditions is exponentially generic.
\begin{enumerate}
\item\label{C0} $|w_i| \ge |\vec w|(1-\e)$ for $1\le i \le k$.

\item\label{C1} $\abs{\overline w_i} \ge (\frac{1-\theta}{\theta} -\e)\abs{w_i} + 2\delta$.

\item\label{C2} $(\overline w_i^{\pm 1} \vert \overline w_j^{\pm 1})_1 < \e \abs{\vec w}$ except when $i=j$ and the exponents are equal.
\end{enumerate}

\end{lem}
\begin{proof}
It suffices to show that for each condition above the set of $\vec w$ which satisfy that condition is exponentially generic. A straightforward counting argument suffices for (\ref{C0}), and (\ref{C1}) follows from Lemma~\ref{amenable}. The remaining assertion follows from Lemmas~\ref{prod2} and~\ref{prod1}.
\end{proof}

Now we complete the proof of Theorem~\ref{hyp}.
\begin{proof}
Pick $\e>0$ as in the statement of Theorem~\ref{hyp}; without loss of generality $\e< 1/3$. We apply Lemma~\ref{C} with $\e'=\e/3$ in place of $\e$ to show that the conclusions of Theorem~\ref{hyp} hold for all $\vec w\in\mathcal C$.

Fix $\vec w=(w_1, \ldots w_k) \in \mathcal C$ and a freely reduced word $z$ in the $w_i$'s. Write $z=x_1\cdots x_t$ where each $x_j$ equals $w_i$ or $w_i\inv$ for some $i$. By~(\ref{eq:Gromov})
\begin{align*}
|\overline x_j - \overline x_{j+1}| &= |\overline x_j| + |\overline x_{j+1}| - 2(\overline x_j | \overline x_{j+1} )_1\\
&\ge \max\{ |\overline x_j|, |\overline x_{j+1}| \} + (\frac{1-\theta}{\theta}-\e')|\vec w|(1-\e') + 2\delta - \e'|\vec w|\\
&\ge \max\{ |\overline x_j|, |\overline x_{j+1}| \} + (\frac{1-\theta}{\theta}-\e)|\vec w| + 2\delta.
\end{align*}
For $1 < \ell \le t$ Lemma~\ref{length} yields
\begin{equation*}
|\overline{x_1\cdots x_{\ell}}| \ge |\overline{x_1\cdots x_{\ell-1}}| + (\frac{1-\theta}{\theta}-\e)|\vec w| \ge \ell (\frac{1-\theta}{\theta}-\e)|\vec w| > 0.
\end{equation*}
Hence $|\overline z| > 0 $, which implies that $Y=\{w_1,\ldots w_k\}$ freely generates a free subgroup $H\subset G$. In addition since $|\overline z|_G = |\overline z| \ge t(\frac{1-\theta}{\theta}-\e)|\vec w|$ and $|\overline z|_{G,Y} \le t|\vec w|$, the compression factor is bounded below by $\frac{1-\theta}{\theta}-\e$.
\end{proof}

The last part of the preceding proof provides the following corollary.

\begin{cor}\label{free}
Let $\mathcal D$ be the set of all $K$-tuples $\vec w=(w_1,\ldots, w_k)$ such that
$|\overline{w_i}^{\pm 1} - \overline{w_j}^{\pm 1}| > \max{|\overline{w_i}|, |\overline{w_j}|} + 2\delta$ except when $i=j$ and the exponents agree. $\mathcal D$ is exponentially generic, and for each $\vec w=(w_1,\ldots,w_k)\in \mathcal D$ and freely reduced word $w_{i_1}^{e_1}\cdots w_{i_t}^{e_t}$,
$|\overline{w_{i_1}^{e_1}\cdots w_{i_\ell}^{e_{\ell}}}| > |\overline{w_{i_1}^{e_1}\cdots w_{i_{\ell-1}}^{e_{\ell-1}}}|$
for $1< \ell \le t$.
\end{cor}
\begin{proof}
 $D$ is exponentially generic because it contains $\mathcal{C}$. By hypothesis there exists $\kappa > 0$ such that
\[|\overline w_j^{\pm 1} - \overline w_{j+1}^{\pm 1}|> \max\{ |\overline w_i|, |\overline w_j| \} + \kappa + 2\delta\]
for all applicable cases. Lemma~\ref{length} applies.
\end{proof}

\section{The membership problem for generic subgroups of hyperbolic groups}\label{membership}

Let $W\to G$ be a choice of generators for $G$. The membership problem is to decide for words $z, w_1,\ldots, w_k\in W$ if $\overline u$ is in the subgroup generated by $\overline{w_1},\ldots \overline{w_k}$. Corollary~\ref{free} provides the basis for a procedure to solve the membership problem once we know how to compute geodesic representatives for $u\in W$, that is words of minimum length with the same image in $G$ as $u$.

There is no uniform algorithm for computing geodesic representatives in presentations of hyperbolic groups. If there were, then since trivial groups are hyperbolic, there would be a feasible procedure to decide whether a finite presentation presents the trivial group; namely check the geodesic length of all the generators. However, this decision problem is unsolvable.

On the other hand given a presentation for a hyperbolic group $G$, one can precompute a strongly geodesic automatic structure for $G$ with respect to the original choice of generators as well as an integer $\delta$ such that all geodesic triangles are $\delta$-thin~\cite{EH}. For the reasons we have discussed there is no computable bound (in terms of the size of the original presentation) for how long this precomputation will take. Nevertheless, once the precomputation is done, one can compute geodesic representatives in linear time by an algorithm due to M. Shapiro~\cite{EH2}.

By Corollary~\ref{free} the following partial algorithm solves the membership problem for all $z\in W$ and  $(w_1,\ldots, w_k)\in \mathcal D$. If in addition $z$ is in the subgroup generated by the $w_1,\ldots, w_k$, the algorithm expresses $z$ as a word in the $w_i$'s.

\begin{alg}\label{gwp1}
{\sc Input} $\vec w = (w_1,\ldots, w_k)\in W$, and $z\in W$\\
{\sc If} the hypothesis of Corollary~\ref{free} does not hold, {\sc Output} ''Failure``\\
{\sc Else While} $|\overline z| > 0$\\
\indent {\sc If} $|\overline{zw_j^e}|< |\overline z|$ for some $j$ and $e=\pm 1$\\
\indent {\sc Then Output} $w_j^e$ and set $z$ equal to a geodesic representative of $zw_j^e$\\
\indent {\sc Else Output} ''Failure`` and halt\\
{\sc Output} ''$z$ is in the subgroup generated by $w_1,\ldots, w_k$``
\end{alg}

Checking the hypothesis of Corollary~\ref{free} requires computing $O(k^2)$ geodesic lengths for words of length at most $2|\vec w|$. There will be no more than $|z|$ passes through the while loop, and during each pass $O(k)$ geodesic representatives are computed for words of length at most $|z|+ |\vec w|$. Thus the time complexity of Algorithm~\ref{gwp1} is $O( (k^2 + k|z|)(2|\vec w| + |z|) = O((k|\vec w| +|z|)^3)$.

\section{The Word Problem for Amenable Groups}

In this section we prove Theorem~\ref{generic}.

\begin{proof}
Let $G$ be a finitely presented amenable group with choice of generators $W \to G$  and unsolvable word problem.  Let $\mathcal A$ be a correct partial algorithm for the word problem in $G$. The input to $\mathcal A$ is a word $w \in W$.   Assume that $D$, the domain of $\mathcal A$, is exponentially generic; i.e., there exists a positive $\rho < 1$ such that
\begin{equation}\label{eq:1}
\frac{\abs{W_n - D}}{\abs{W_n}} \le \rho^n \mbox{ for $n$ large enough}
\end{equation}
where $W_n$ is the set of words of length $n$. We shall obtain a contradiction by showing that under these conditions the word problem for $G$ is solvable.

Let $\mathcal B$ be the partial algorithm which on input $w$ recursively enumerates all words $v_1, v_2, \ldots$ defining the identity in $G$ and applies $\mathcal A$ to $wv_1, wv_2, \ldots$. Since $\mathcal A$ does not always converge, we organize this computation as follows. For each $m=1,2,\ldots$, $\mathcal B$ computes $v_1,\ldots v_m$, and applies the first $m$ steps of $\mathcal A$ to each $wv_i$. If $\mathcal A$ halts for some $i$, then eventually $\mathcal B$ discovers that fact and halts too. $\mathcal B$ accepts $w$ as a word defining the identity if and only if $\mathcal A$ accepts $wv_i$.

Clearly $\mathcal B$ converges on $w$ if and only if $\mathcal A$ converges on some $wv_i$. Hence there must exist a word $w$ such that $\mathcal A $ does not halt on any $wv_i$. Fix $n > \abs w$.  For any $v_i$ of length $n-\abs w$, we have $wv_i\in W_n -D$ because $\mathcal A$ does not halt on $wv_i$. We conclude
\[
\abs{W_n - D} \ge \abs{V_{n-\abs w}}
\]
where for any $k\ge 0$, $V_k$ is the set of  words of length $k$ which define the identity in $G$. Since $G$ is amenable, we have
\[
\lim_{k\to\infty, 2\vert k}\bigl(\frac{\abs{V_k}}{\abs{W_k}}\bigr)^{1/k} = 1.
\]

It follows from the equations above that for $n$ large enough and even,
\[
\abs{V_{n-\abs w}} \ge (\frac{1+\rho}{2})^{(n-\abs w)}\abs{W_{(n-\abs w)}},
\]
whence
\[
\abs{W_n-D} \ge (\frac{1+\rho}{2})^{(n-\abs w)}\abs{W_{(n-\abs w)}}
\]
which implies
\[
\rho^n \geq \frac{\abs{W_n - D}}{\abs{W_n}}\ge C_w(\frac{1+\rho}{2})^n
\]
for some constant $C_w$ (depending on $w$)  and infinitely many $n$,  which is impossible since $\rho < \frac{1+\rho}{2}$.
\end{proof}

\section{Open Problems}

In this section we formulate some  open problems which seem to be interesting in this area. 

Let $G$ be a group generated by a finite set $A$ and  $W$ the free monoid with basis $A \cup A^{-1}$.   For $k\ge 1$ put
$
\W k = \{(w_1,\ldots, w_k) \mid w_i \in W\}
$
and define the disk of radius $n$ in $\W k$ by $\W k_n  =  \{(w_1, \ldots,w_k)  \in\W k\mid \abs{w_i} \le n\}$.

We say that a group $G$  satisfies the {\em generic free basis} property if for each choice of generators $W\to G$  and every $k\ge 1$ the set  of all tuples $(w_1, \ldots, w_k)\in W^{(k)}$ for which $\overline w_1, \ldots, \overline w_k$ generate a free subgroup of rank $k$ in $G$, is generic  with respect to the stratification  $W^{(k)} = \cup_{n =1}^\infty \W k_n$.

\begin{problem}
Does a finitely generated group $G$ have the generic free basis property if some of its subgroups of finite index has it?
\end{problem}

\begin{problem}
Does the group $SL(n,\mathbb{Z})$, $n \geq 3$, have the generic free basis property?

\end{problem}

\begin{problem}
Construct a finitely presented group where the word problem is undecidable on every generic set of inputs (which are words in a given finite generating set).
\end{problem}


%

\end{document}